\documentclass[11pt, letterpaper]{amsart}
\usepackage{amsfonts,amsthm,mathtools,mathrsfs}
\usepackage[all,cmtip]{xy}
\usepackage{fullpage}
\usepackage{amsmath}
\usepackage{amssymb}
\usepackage{enumerate}
\usepackage{tikz}
\usepackage[backref]{hyperref}
\usepackage{cleveref}
\usepackage{verbatim}
\usepackage{cancel}
\usepackage{float}

\newtheorem{question}{Question}

\newtheorem{lem}{Lemma}[section]
\newtheorem{thm}[lem]{Theorem}
\newtheorem{theorem}[lem]{Theorem}

\newtheorem{cor}[lem]{Corollary}

\theoremstyle{definition}
\newtheorem{remark}[lem]{Remark}

\DeclareMathAlphabet{\curly}{U}{rsfs}{m}{n}
\newcommand{\ra}{\ensuremath{\rightarrow}}
\newcommand{\tors}{\operatorname{tors}}
\newcommand{\End}{\operatorname{End}}
\newcommand{\Aut}{\operatorname{Aut}}

\newcommand{\ord}{\operatorname{ord}}

\newcommand{\Q}{\mathbb{Q}}

\newcommand{\Z}{\mathbb{Z}}

\newcommand{\CM}{\operatorname{CM}}

\newcommand{\OO}{\mathcal{O}}
\newcommand{\ff}{\mathfrak{f}}

\mathchardef\mhyphen="2D

\title{Torsion for CM Elliptic Curves Defined over Number Fields of Degree $2p$}

\author{Abbey Bourdon}

\author{Holly Paige Chaos}

\begin{document}


\begin{abstract}
For a prime number $p$, we characterize the groups that may arise as torsion subgroups of an elliptic curve with complex multiplication defined over a number field of degree $2p$. In particular, our work shows that a classification in the strongest sense is tied to determining whether there exist infinitely many Sophie Germain primes.
\end{abstract}

\maketitle

\section{Introduction}
 In 1922, Mordell proved that the set of $\Q$-rational points of an elliptic curve $E$ defined over $\Q$ is a finitely generated abelian group \cite{mordell22}. That is, 
 $E(\mathbb{Q})\cong E(\mathbb{Q})_{\tors}\times \mathbb{Z}^r$, where $E(\mathbb{Q})_{\tors}$ denotes the finite set of torsion points and $r\in \Z^{\geq 0}$ is the rank of $E/\Q$. It is natural to ask what groups arise as $E(\mathbb{Q})_{\tors}$ as $E$ ranges over all elliptic curves over $\Q$, and the answer is known due to work of Mazur.

 \begin{theorem}[Mazur, \cite{mazur77}]
Let $E/\Q$ be an elliptic curve. Then $E(\Q)_{\tors}$ is isomorphic to one of the following groups:
\begin{center}\begin{tabular}{ll}
$\Z/m\Z$ &$1 \leq m \leq 10$ or $m=12$\\
$\Z/2\Z \times \Z/2m\Z$ &$1 \leq m \leq 4.$
\end{tabular}
\end{center}
Furthermore, each of these groups occurs as a torsion subgroup of an elliptic curve $E/\Q.$
 \end{theorem}

More generally, if $E$ is an elliptic curve defined over a number field $F$, then the set of $F$-rational points of $E$ is again a finitely generated abelian group by Weil \cite{weil29}, so one may seek to classify the groups occurring as $E(F)_{\tors}$. In fact, by Merel's Uniform Boundedness Theorem \cite{merel}, there are only finitely many groups that arise as $E(F)_{\tors}$, even as $E$ ranges over all elliptic curves defined over all number fields $F$ of a fixed degree. Thus the fundamental question which motivates our work is the following:

\begin{question}
For a fixed $d \in \Z^+$, what groups arise as torsion subgroups of an elliptic curve defined over a number field of degree $d$?
\end{question}

Now 100 years after Mordell's proof, the answer to Question 1 is known only for $d \leq 3$; see \cite{mazur77,kamienny86,KM88,kamienny92,Deg3Class}. A fundamental obstruction to extending the classification to $d>3$ is the existence of so-called \textbf{sporadic} or \textbf{isolated} points on modular curves which can give rise to torsion subgroups occurring on only finitely many elliptic curves (up to isomorphism) defined over all number fields of a fixed degree. To date, we lack adequate tools for detecting such points, and hence the problem of classifying torsion subgroups of elliptic curves over higher degree number fields remains largely open.

One way to obtain classification results beyond cubic fields is to restrict the elliptic curves under consideration. One common family of elliptic curves to study in this context is elliptic curves $E/\mathbb{Q}$ under base extension, where the classification of torsion subgroups is known for degrees $d\leq 5,$ $d=7$, or $d$ not divisible by a prime $\leq 7$; see \cite{najman16,GJNajman,GJ17}. If we require only that the $j$-invariant of $E$ lie in $\Q$, then analogous classification results exist \cite{Guzvic,CremonaNajmanQCurve}. Another common family is elliptic curves with \textbf{complex multiplication (CM)}, which are elliptic curves with unusually large endomorphism rings. Whereas most elliptic curves have endomorphism ring isomorphic to $\Z$, we say $E/F$ is a CM elliptic curve if $\End_{\overline{F}}(E) \cong \OO$, an order in an imaginary quadratic field $K$. Each order is uniquely determined by its discriminant $\Delta\coloneqq [\OO_K:\OO]^2 \cdot \Delta_K$, where  $\Delta_K$ is the discriminant of $K$ and $\OO_K$ is its ring of integers. For the set of all CM elliptic curves, the classification of torsion subgroups is known for any $d\leq 13$ or for any odd $d>13$; see \cite{tor2,Olson74,BP,BCS}. We note that CM elliptic curves produce many examples of sporadic points on modular curves (see, for example, \cite{LeastCMDeg}), so this provides further motivation for studying this class in particular.

 In the present work, we extend the classification of torsion subgroups of CM elliptic curves to those defined over any number field of degree twice a prime, building on work of the first author and Clark \cite{BC1,BC2}. In fact, since the classification is known for $d=4,6,$ and 10 by \cite{tor2}, we need only consider fields of degree $2p$ for primes $p>5$.
 Our classification is most clearly stated in the context of new subgroups. By Theorem 2.1 in \cite{BCS}, if a torsion subgroup arises in degree $d',$ then it arises in any degree $d$ for which $d'\mid d.$ We say a CM torsion subgroup is \textbf{new} if it occurs in degree $d$ and not in any degree $d' < d$ such that $d'\mid d.$ Since torsion subgroups of CM elliptic curves in degrees 1 and 2 are known \cite{Olson74,tor2}, and there are no new CM torsion subgroups in degree $p>5$ for $p$ prime \cite{BCS}, it suffices to classify only the new subgroups arising in degree $2p$. 

\begin{theorem}\label{chaos}
Let $F$ be a number field of degree $2p$ for $p>5$ prime and let $E/F$ be an elliptic curve with CM by the order of discriminant $\Delta$. Then $E(F)_{\tors}$ is new if and only if one of the following occurs: 
\begin{enumerate} 
\item $\Delta=-115$, $p=11$, and $E(F)_{\tors}\cong \Z/23\Z$.
\item $\Delta=-235$, $p=23$, and $E(F)_{\tors}\cong \Z/47\Z$.
\item $\Delta \in\{-11,-19,-27,-43,-67,-163\}$, $2p+1$ is prime with $\left( \frac{\Delta}{2p+1} \right) = 1$, and \\$E(F)_{\tors}\cong \Z/(2p+1)\Z$.
\item $\Delta  \in\{-8,-12,-16,-28\}$, $2p+1$ is prime with $\left( \frac{\Delta}{2p+1} \right) = 1$, and\\ $E(F)_{\tors}\cong \Z/2(2p+1)\Z$.
\item $\Delta=-7$, $2p+1$ is prime with $\left( \frac{\Delta}{2p+1} \right) = 1$, and $E(F)_{\tors}\cong \Z/2\Z \times \Z/2(2p+1)\Z$.
\item $\Delta=-3$, $p=7$, and $E(F)_{\tors}\cong \Z/49\Z$.
\item $\Delta=-3$, $6p+1$ is prime, and $E(F)_{\tors}\cong \Z/(6p+1)\Z$.
\item $\Delta=-4$, $4p+1$ is prime, and $E(F)_{\tors}\cong \Z/2(4p+1)\Z$.

\end{enumerate}
In particular, any new torsion subgroup arises on one of only finitely many CM elliptic curves, and all but $\Delta=-115$ and $-235$ correspond to imaginary quadratic orders of class number 1.
\end{theorem} 

\begin{remark}
In \cite{GuzvicRoy}, the authors classify torsion subgroups of Mordell curves defined over $\Q$ under base extension to number fields of degree $2p$ and $3p$, where $p \geq 5$ is prime. Every Mordell curve $E$ has $j(E)=0$ and CM by the order of discriminant $\Delta=-3$. Our classification result includes additional groups since we are not requiring elliptic curves with $j(E)=0$ to be defined over $\Q$.
\end{remark}

Theorem \ref{chaos} tells us that if $\Delta \neq -3, -4$, then the only $\Delta$-CM torsion subgroups that can arise in degree $2p$ for $p>5$ that did not occur over a number field of degree 2 or degree $p$ must have exponent $2p+1$ or $2(2p+1)$, where $p$ is a Sophie Germain prime. It is conjectured that there are infinitely many Sophie Germain primes, though this remains unproven. These primes were a vital piece of Sophie Germain's investigations concerning Fermat's Last Theorem.

 From Theorem \ref{chaos}, we can quickly deduce the torsion subgroups that arise for CM elliptic curves defined over number fields of degree $2p$ where $p>5$ is prime, including for the first previously unknown degree $d=14$. For example, 7 is not a Sophie Germain prime, but $6\cdot7+1$ and $4\cdot7+1$ are both prime. Thus, by Theorem \ref{chaos}, the new torsion subgroups in degree 14 are precisely $\Z/43\Z,$ $\Z/49\Z,$ and $\Z/58\Z$. We record this and other small degrees in the following result.
  
 \begin{cor}\label{fourteen} Let $F$ be a number field of degree $2p$ for $p \in \{7,11,13,17,19\}$, and let $E/F$ be a CM elliptic curve. The group $E(F)_{\tors}$ is isomorphic to one of the following groups which arises over quadratic fields
\begin{center}\begin{tabular}{ll}
$\Z/m\Z$ &for $m=1,2,3,4,6,7,$ or $10$\\
$\Z/2\Z \times \Z/2m\Z$ & for $m=1,2$, or $3$ \\
$\Z/3\Z \times \Z/3\Z$
\end{tabular}
\end{center}
or else
\begin{enumerate}
\item $p=7$ and $E(F)_{\tors} \cong \Z/m\Z$ for $m=43,49,$ or $58$,
\item $p=11$ and $E(F)_{\tors} \cong \Z/m\Z$ for $m=23,46,67$ or $\Z/2\Z \times \Z/46\Z$,
\item $p=13$ and $E(F)_{\tors} \cong \Z/m\Z$ for $m=79$ or $106$, or
\item $p=17$ and $E(F)_{\tors} \cong \Z/m\Z$ for $m=103$.
\end{enumerate}
Moreover, each group occurs.
\end{cor}
\begin{remark}
Since there are only finitely many CM $j$-invariants contained in all number fields of a fixed degree (see $\S 2$), each of these groups necessarily arises on only finitely many CM elliptic curves.
\end{remark}


In particular, by Corollary \ref{fourteen}, we see that \emph{no} new torsion subgroups arise on CM elliptic curves defined over number fields of degree $d=2\cdot 19$. Thus, another consequence of Theorem \ref{chaos} is a description of degrees of the form $2p$ such that no new torsion subgroups occur.

     \begin{cor}\label{chaos22} Let $F$ be a number field of degree $2p$, for $p>5$, and suppose none of the following hold:
     \begin{enumerate}
     \item $2p+1$ is prime and split in an imaginary quadratic order of class number 1 with $\Delta<-4$.
     \item $4p+1$ is prime.
     \item $6p+1$ is prime.
     \end{enumerate} Then for any CM elliptic curve $E/F$, the torsion subgroup $E(F)_{\tors}$ is isomorphic to one of the groups that arise for CM elliptic curves defined over quadratic fields. 
\end{cor}
This finding is significant in the context of ``stratification of torsion," a phenomenon first explored in \cite{BCP,BP} for CM torsion subgroups in odd degree. For any positive integer $d$, let $\mathscr{G}_{\CM}(d)$ denote the set of isomorphism classes of groups which arise as $E(F)_{\tors}$ for some CM elliptic curve $E$ over some degree $d$ number field $F$. For any positive integer $d$, we define the set of $d$-Olson degrees to be those positive integers $d'$ for which $\mathscr{G}_{\CM}(d')=\mathscr{G}_{\CM}(d)$. In the case of odd $d$, we find that the set of $d$-Olson degrees possesses a positive asymptotic density \cite{BP}, but whether the same holds true for any even $d$ is still an open problem. See \cite[Questions 1.6]{BP}. 

\begin{remark}
In fact, as noted by Clark, Corollary \ref{chaos22} implies there exist infinitely many 2-Olson degrees. Recall the Prime Number Theorem states that the number of primes $p \leq X$ is asymptotic to $\frac{X}{\log X}$. On the other hand,
for any even $a \in \Z^+$, as $X \ra \infty$ the number of primes $p \leq X$ such that $ap+1$ is also prime is $O(\frac{X}{\log^2 X})$; see \cite[Thm. 3.12]{SieveMethods}.  
By applying this with $a = 2, 4$ and $6$, we see that there are infinitely many primes $p \leq X$ such that $2p$ is a $2$-Olson degree. 
\end{remark}

\section*{Acknowledgements}

We thank Frank Moore, Jeremy Rouse, and the anonymous referees for helpful comments on an earlier draft, and we thank Pete L. Clark for Remark 1.7 and other helpful comments. The first author was partially supported by an A.J. Sterge Faculty Fellowship and NSF grant DMS-2137659.

\section{Background and Notation}

For most elliptic curves $E$ over a number field $F$, the ring of endomorphisms of $E$ defined over $\overline{F}$ is isomorphic to $\Z$, where $n \in \Z$ corresponds to the multiplication-by-$n$ map on $E$. We say an elliptic curve has \textbf{complex multiplication}, or CM, if its endomorphism ring is strictly larger than $\mathbb{Z}.$ For a CM elliptic curve $E/F$, there is an imaginary quadratic field $K$ and positive integer $f$ such that End$_{\bar{F}}(E) \cong \mathcal{O}=\mathbb{Z}+f\mathcal{O}_K,$ the order in $K$ of \textbf{conductor} $f$. Here $\OO_K$ denotes the full ring of integers in $K$. In particular, we note that $\OO \subseteq \OO_K$ and $[\OO_K:\OO]=f$. The order is largest when $f=1$, and so we call $\OO_K$ the \textbf{maximal order}.
Any order $\OO$ in an imaginary quadratic field $K$ can be uniquely identified using its \textbf{discriminant}, \[\Delta=\Delta(\OO)=f^2\cdot \Delta_K,\] where $\Delta_K$ is the discriminant of $K$. We let $\omega$ denote the number of units in $\OO$, so
\[ \omega = \begin{cases} 6 & \text{if }\Delta = -3, \\
4 & \text{if }\Delta = -4, \\ 2 & \text{if } \Delta < -4.
\end{cases} \]
For an elliptic curve $E$ with CM by the order of discriminant $\Delta$, we have $\Delta=-3$ if and only if $j(E)=0$ and $\Delta=-4$ if and only if $j(E)=1728$. We use $w_K$ to denote $\#\OO_K^{\times}$.

CM elliptic curves have a well-known and beautiful connection with class field theory. For example, if $E$ has CM by the maximal order in $K$, then $K(j(E),\mathfrak{h}(E_{\tors}))$ is the maximal abelian extension of $K$, where $\mathfrak{h}: E \rightarrow E/\Aut(E)\cong \mathbb{P}^1$ denotes a Weber function on $E$. If one adjoins the values of a Weber function only on points of order dividing $N$, we obtain the ray class field of $K$ modulo $N$; see, for example Theorem II.5.6 and Corollary II.5.7 of \cite{advanced}. Of particular relevance to the present work is the fact that if $E$ has CM by the order in $K$ of conductor $f$, then $K(j(E))$ is the ring class field of $K$ of conductor $f$ and $[K(j(E)):K]=[\mathbb{Q}(j(E)):\mathbb{Q}]=h(\OO)$, the class number of $\OO$.
For an elliptic curve $E$ with CM by the order of discriminant $f^2\Delta_K$ with $f \geq 2$, we have \cite[Cor. 7.24]{cox}
\begin{equation}
\label{RINGCLASSDEGEQ}
[K(j(E)):K] =h_K \frac{2}{w_K} f \prod_{p \mid f} \left( 1- \left(\frac{\Delta_K}{p}\right) \frac{1}{p} \right),
\end{equation}
where $h_K$ denotes the class number of $K$ and $\left(\frac{\Delta_K}{p}\right)$ is the Kronecker symbol. As there are only finitely many imaginary quadratic fields of a given class number \cite[Theorem III]{heilbronn}, there are only finitely many imaginary quadratic orders of a given class number by (1). For each imaginary quadratic order $\OO$, there are precisely $h(\OO)$ non-isomorphic $\OO$-CM elliptic curves. 

A crucial ingredient in the proof of our main result is the following theorem. Recall $\omega=\#\mathcal{O}^{\times}$.
  \begin{theorem}[Bourdon, Clark, { \cite[Theorem 4.1]{BC2}}] \label{bigmnthm}  Let $K$ be an imaginary quadratic field, and let $\OO$ be the order in $K$ of conductor $f$.
Let $M = \ell_1^{a_1} \cdots \ell_r^{a_r} \mid N = \ell_1^{b_1} \cdots \ell_r^{b_r}$ where $\ell_1< \dots <\ell_r$ are prime numbers and $a_i$, $b_i$ are nonnegative integers. 
\begin{enumerate}
\item There is $T(\OO,M,N) \in \Z^+$ such that: for all $d \in \Z^+$, there is a number field $F \supset K(j(E))$ 
such that
$[F:K(j(E))] = d$ and an $\OO$-CM elliptic curve $E/F$ such that $\Z/M\Z \times \Z/N \Z \hookrightarrow E(F)$ if and only if $T(\OO,M,N) \mid d$.  
\item 
If $N=2$ or 3, then $T(\OO,M,N)$ is as follows:
\begin{align*}
T(\OO, 1,2)&=\begin{cases} 3& \left( \frac{\Delta}{2} \right)=-1 \text{ and } \Delta \neq -3 \\ 1 & \text{otherwise} \end{cases},\\
T(\OO,1,3) &= \begin{cases} 8/\omega& \left( \frac{\Delta}{3} \right)=-1  \\ 1 & \text{otherwise} \end{cases},\\
 T(\OO,2,2) &= \frac{2(2-\left( \frac{\Delta}{2} \right))}{\omega},\\
 T(\OO,3,3) &= \frac{2(3-\left( \frac{\Delta}{3} \right))}{\omega}.
  \end{align*}
  
  \item Suppose $N \geq 4$.  Then we have
\[ T(\OO,M,N) = \frac{\prod_{i=1}^r \widetilde{T}(\OO,\ell_i^{a_i},\ell_i^{b_i})}{\omega} \]
where the definition of $\widetilde{T}(\OO,\ell^a,\ell^b)$ appears below. Put $c := \ord_{\ell}(f)$. 

\begin{enumerate}
\item[i)] If $\left( \frac{\Delta}{\ell} \right) = -1$, then
\[ \widetilde{T}(\OO,\ell^a,\ell^b) \coloneqq \ell^{2b-2}(\ell^2-1). \]
\item[ii)] If $\left( \frac{\Delta}{\ell} \right) = 1$, then
\[ \widetilde{T}(\OO,\ell^a,\ell^b) \coloneqq \begin{cases} \ell^{b-1}(\ell-1) & a = 0 \\ \ell^{a+b-2} (\ell-1)^2 & a \geq 1 \end{cases}. \]
\item[iii)] If $\ell \mid \ff $ and $\left( \frac{\Delta_K}{\ell} \right) = 1$, then
\[ \widetilde{T}(\OO,\ell^a,\ell^b) \coloneqq \ell^{a+b-1}(\ell-1). \]
\item[iv)] If $\left( \frac{\Delta_K}{\ell} \right) = 0$, then
\[ \widetilde{T}(\OO,\ell^a,\ell^b) \coloneqq\begin{cases} \ell^{a+b-1}(\ell-1) & b \leq 2c+1 \\
\ell^{\max(a+b-1,2b-2c-2)}(\ell-1) & b > 2c+1 \end{cases}. \]
\item[v)] If $\ell \mid f$ and $\left( \frac{\Delta_K}{\ell} \right) = -1$, then
\[ \widetilde{T}(\OO,\ell^a,\ell^b) \coloneqq \begin{cases} \ell^{a+b-1}(\ell-1) & b \leq 2c \\
\ell^{\max(a+b-1,2b-2c-1)}(\ell-1) & b > 2c \end{cases}. \]

\end{enumerate}
\end{enumerate}
\end{theorem}

From this, we deduce the following corollary, which also appears as Theorem 6.2 in \cite{BC1}. It refines earlier results of Silverberg \cite{silverberg88,silverberg92}.

\begin{cor}
\label{dividethm}
Let $\OO$ be an order in an imaginary quadratic field $K$, and let $N \in \Z^+$. Then \[\varphi(N) \mid \omega \cdot T(\OO,1,N).\] \end{cor}

Suppose $E/F$ is an $\OO$-CM elliptic curve with $\Z/M\Z \times \Z/N\Z \hookrightarrow E(F)$. Since $[K(j(E)):K]=[\mathbb{Q}(j(E)):\mathbb{Q}]=h(\OO)$, we can actually consider the divisibility conditions in Theorem \ref{bigmnthm} over  $\Q(j(E))$, as illustrated in the field diagram below. 
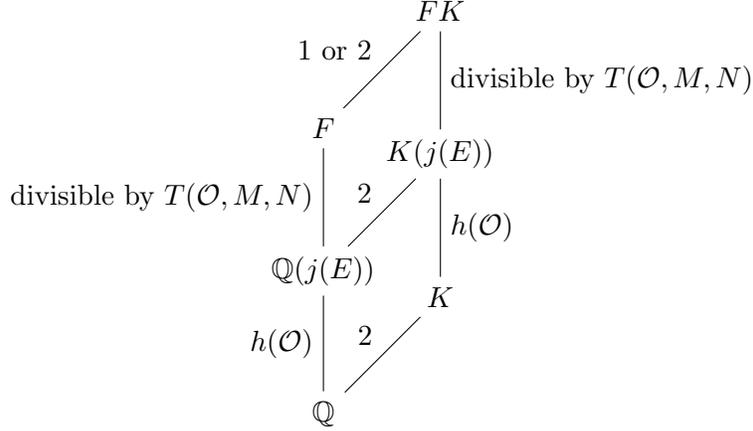
\begin{center}
\begin{tikzpicture}[node distance=1.8cm]
\node (Q)                  {$\Q$};
\node (Qj) [above of=Q, node distance=1.9cm] {$\Q(j(E))$};
\node (K) [above right of =Q, node distance=2.2 cm] {$K$};
\node (QKj)  [above of=K, node distance=1.9 cm]   {$K(j(E))$};
\node (F) [above of=Qj, node distance=1.9cm] {$F$};
\node (FK)  [above of=QKj, node distance=1.9 cm]   {$FK$};

 \draw[-] (Q) edge node[auto]{$h(\OO)$} (Qj);
 \draw[-] (Q) edge node[auto] {$2$} (K);
 \draw[-] (K) edge node[right] {$h(\OO)$} (QKj);
 \draw[-] (Qj) edge node[auto] {$2$} (QKj);
 \draw[-] (Qj) edge node[auto] {divisible by $T(\OO,M,N)$} (F);
 \draw[-] (QKj) edge node[right] {divisible by $T(\OO,M,N)$} (FK);
  \draw[-] (F) edge node[auto] {$1$ or $2$} (FK);

\end{tikzpicture}
\end{center}

\begin{cor}
\label{dividethmGen}
Let $\OO$ be an order in an imaginary quadratic field $K$, and let $E/F$ be an $\OO$-CM elliptic curve with an $F$-rational point of order $N \in \Z^+$. Then $T(\OO, 1,N) \mid [F:\Q(j(E))]$ and \[\varphi(N) \mid \omega \cdot [F:\Q(j(E))].\] \end{cor}

\begin{proof}
This follows from Corollary \ref{dividethm} and the diagram above. 
\end{proof}

Following \cite{BC2}, for any imaginary quadratic order $\OO$ and integers $M \mid N$, we let $T^{\circ}(\OO,M,N)$ denote the least degree of an extension $F/\Q(j(E))$ in which an $\OO$-CM elliptic curve $E/F$ has $\Z/M\Z \times \Z/N\Z \hookrightarrow E(F)$. In particular, $F$ need not contain the CM field $K$. We note  $T^{\circ}(\OO,M,N)=2^{\epsilon}\cdot T(\OO,M,N)$, where $\epsilon \in \{0,1\}$. Explicit formulas for $T^{\circ}(\OO,M,N)$ for fixed $\OO$ are computed in \cite[$\S8$]{BC2}.

In the case where $M=1$, we use the streamlined notation $T(\OO,N) \coloneqq T(\OO,1,N)$ and $T^{\circ}(\OO,N) \coloneqq T^{\circ}(\OO,1,N)$. We have the following description of $T^{\circ}(\OO,N)$, which follows from Theorems 1.3, 6.1, 6.2, and 6.6 in \cite{BC2}.

\begin{thm}[Bourdon, Clark \cite{BC2}] \label{LeastDegN}
Let $\OO$ be an imaginary quadratic order of conductor $f$ in $K$.  Let $N \in \Z^+$ have prime power decomposition 
$\ell_1^{a_1} \cdots \ell_r^{a_r}$ with $\ell_1 < \ldots < \ell_r$. The least degree over $\Q(j(E))$ in which there is an $\OO$-CM elliptic curve $E$ with a rational point of order $N$ is $T(\OO,N)$ if and only if $T^{\circ}(\OO,\ell_i^{a_i})=T(\OO,\ell_i^{a_i})$ for all $1 \leq i \leq r$. Otherwise the least degree is $2\cdot T(\OO,N)$. Moreover, $T^{\circ}(\OO,\ell_i^{a_i})=T(\OO,\ell_i^{a_i})$ if and only if one of the following holds, where $c_i \coloneqq \ord_{\ell_i}(f)$:
\begin{enumerate}
\item $\ell_i$ is inert in $\OO$
\item $\ell_i^{a_i}=2$ and is split or ramified in $\OO$
\item $\ell_i^{a_i}=2^{a_i}$ where $2$ is ramified in $\OO$ but not in $K$, $c_i \geq 2$, and $a_i  \leq 2c_i-2$  
\item $\ell_i^{a_i}=2^{a_i}$ where $2$ is ramified in $K$ and $c_i=0$ 
\item $\ell_i^{a_i}=2^{a_i}$ where $\ord_2(\Delta_K)=2$, $c_i \geq 1$, and $a_i  \leq 2c_i$  
\item $\ell_i^{a_i}=2^{a_i}$ where $\ord_2(\Delta_K)=3$, $c_i \geq 1$
\item $\ell_i>2$ is ramified in $\OO$ but split in $K$ and $a_i \leq 2c_i$
\item $\ell_i>2$ is ramified in $\OO$ and not split in $K$ 
\end{enumerate}
\end{thm}

Let $E/F$ be an $\OO$-CM elliptic curve and $P \in E$ a point of order $N$. If $[F:\Q]=T^{\circ}(\OO,N)\cdot h(\OO)$, then $F=\Q(j(E),\mathfrak{h}(P))$, where $\mathfrak{h}: E \rightarrow E/\Aut(E)\cong \mathbb{P}^1$ is a Weber function on $E$. Moreover, if $\psi:E \rightarrow E'$ is an isomorphism, then $\mathfrak{h}(P)=\mathfrak{h}(\psi(P))$ by \cite[p.107]{shimura}. It follows that for any $P \in E$, the fields $K(j(E),\mathfrak{h}(P))$ and $\Q(j(E),\mathfrak{h}(P))$ do not depend on the chosen Weierstrass equation for $E$. See \cite[$\S2.4$]{BC2} and \cite[$\S7A$]{BC1} for additional details.

\section{Determining the Exponent of New Subgroups}

Let $p>5$ be a prime number, and suppose $F$ is a number field of degree $2p$. Let $E/F$ be a CM elliptic curve with $E(F)_{\tors} \cong \Z/M \Z \times \Z/N \Z$ for $M \mid N$. By definition, this torsion subgroup is \textbf{new} if it does not occur as the torsion subgroup of a CM elliptic curve defined over a number field of degree $1$, $2$, or $p$. However, every CM torsion subgroup arising in degree 1 also arises in degree 2, and there are \emph{no} new torsion subgroups of CM elliptic curves in prime degree $p>5$ by \cite[Theorem 1.4]{BCS}. Thus $E(F)_{\tors}$ is new if and only if it does not occur in degree 2.

In this section, we will determine the possible exponents of a new CM torsion subgroup in degree $2p$. 
If $E(F)_{\tors}\cong \Z/M \Z \times \Z/N \Z$ is a new torsion subgroup, then either $N$ appears already as the exponent of a CM torsion subgroup in degree 2 and $N \in \{1,2,3,4,6,7,10\}$ by \cite[$\S4.2$]{tor2}, or else it has exponent outside this list. We say $E(F)_{\tors}$ has a \textbf{new exponent $N$} if $E(F)_{\tors}$ is new and $N \not\in\{1,2,3,4,6,7,10\}$.

\subsection{Two Preliminary Lemmas} By Corollary \ref{dividethm}, if $\OO$ is an order in an imaginary quadratic field and $N \in \Z^+$, then
\[
\varphi(N) \mid \omega \cdot T(\OO,N).
\]
Since $T^{\circ}(\OO,N) \in \{T(\OO,N), 2\cdot T(\OO,N)\}$, this implies $\varphi(N) \mid \omega \cdot T^{\circ}(\OO,N)$. The following lemma shows equality can hold under only very specific conditions.
\begin{lem} \label{BabyLem}
Let $N \in \Z^{\geq 4}$ have prime power decomposition $\ell_1^{a_1} \cdots \ell_r^{a_r}$ with $\ell_1 < \cdots <\ell_r$, and let $\OO$ be an imaginary quadratic order of discriminant $\Delta$. If $\varphi(N)= \omega \cdot T^{\circ}(\OO,N)$, then every $\ell_i$ with $\ell_i^{a_i}\geq3$ is ramified in $\OO$. If $\ell_i^{a_i}=2$, then $2$ is split or ramified in $\OO$.
\end{lem}

\begin{proof}
Suppose $\ell \mid N$ is prime and $\ord_{\ell}(N)=a$ with $\ell^a \geq 3$, and suppose $\varphi(N)= \omega \cdot T^{\circ}(\OO,N)$. In particular, this implies $T(\OO,N)=T^{\circ}(\OO,N)$ by Corollary \ref{dividethm}, and so by Theorem \ref{LeastDegN}, we have $T(\OO,\ell^a)=T^{\circ}(\OO,\ell^a)$. Then $\left(\frac{\Delta}{\ell}\right)\neq 1$ by Theorem \ref{LeastDegN}.
Suppose $\left(\frac{\Delta}{\ell}\right)=-1.$ Recall from Theorem \ref{bigmnthm} that since $N \geq 4$,
\[ \omega \cdot T(\OO,N) = \displaystyle\prod_{i=1}^r \ \widetilde{T}(\OO,\ell_i^{a_i}). \] If $\varphi(N)=\omega \cdot T(\OO,N)$, we must have $\varphi(N)=\displaystyle\prod_{i=1}^r \ \widetilde{T}(\OO,\ell_i^{a_i}).$ Moreover, since $\varphi(\ell_i^{a_i}) \mid \widetilde{T}(\OO,\ell_i^{a_i})$ for all $i$, we must have $\varphi(\ell^a)=\widetilde{T}(\OO,\ell^{a}).$  By Theorem \ref{bigmnthm} we have \[\widetilde{T}(\mathcal{O},\ell^a)=\ell^{2a-2}(\ell^2-1)=(\ell^{a-1})(\ell^{a-1}) (\ell-1)(\ell+1)>\varphi(\ell^a).\]
We have reached a contradiction. The same kind of calculation shows 2 cannot be inert in $\OO$.
\end{proof}

\begin{lem}\label{Lemma3.2}
Let $E/F$ be an $\OO$-CM elliptic curve with an $F$-rational point of order $N$ for $N \in \Z^+$. If $\omega \cdot [F:\Q(j(E))]=\varphi(N)$, then 
\[
T(\OO,N)=T^{\circ}(\OO,N)=[F:\Q(j(E))].
\]
\end{lem}
\begin{proof}
By Corollaries \ref{dividethm} and \ref{dividethmGen} we have
\[
\varphi(N) \mid \omega \cdot T(\OO,N) \mid \omega \cdot [F:\Q(j(E))] = \varphi(N),
\]
from which we conclude equality holds throughout. Thus $T(\OO,N)=T^{\circ}(\OO,N)=[F:\Q(j(E))]$.
\end{proof}

\subsection{Determining new exponents}
Let $F$ be a number field of degree $2p$, for $p>5$ prime. If $E/F$ is an $\OO$-CM elliptic curve with a point of order $N$, then $h(\OO)=[\Q(j(E)):\Q] \in \{1,2,p, 2p\}$. We will consider each case separately in a series of lemmas. One important ingredient is the following theorem of Parish.
\begin{thm}[Parish, {\cite[$\S6$]{Parish89}}] \label{ParishThm}
Let $E$ be a CM elliptic curve defined over $F=\Q(j(E))$. Then $E(F)_{\tors}$ is isomorphic to one of the following groups: the trivial group $\{\cdot\}, \Z/2\Z, \Z/3\Z, \Z/4\Z, \Z/6\Z$, or $\Z/2\Z\times \Z/2\Z$.
\end{thm}

\noindent All other cases build upon Theorems \ref{bigmnthm} and \ref{LeastDegN} in combination with Lemma \ref{BabyLem}.

\begin{lem}
Let $F$ be a number field of degree $2p$. Suppose $E/F$ is an $\OO$-CM elliptic curve, where $h(\OO)=2p$. Then $E(F)_{\tors}$ is not new.
\end{lem}
\begin{proof}
Here, $F=\Q(j(E))$ and $E(F)_{\tors}$ is one of the groups arising over $\Q$ by Theorem \ref{ParishThm}.
\end{proof}

\begin{lem} \label{Lem3.3}
Let $F$ be a number field of degree $2p$ for $p>5$. Suppose $E/F$ is an $\OO$-CM elliptic curve, where $h(\OO)=p$. Then $E(F)_{\tors}\cong \Z/M\Z\times\Z/N\Z$ for $M \mid N$ and $N \in \{1,2,3,4,6\}$.
\end{lem}

\begin{proof}
Note in this case $[F:\Q(j(E))] = 2$, which means $\varphi(N) \mid \omega \cdot 2$ by Corollary \ref{dividethmGen}. As $h(\OO)=p$, we have $\omega=2$, and so $N \in \{1, 2, 3, 4, 5, 6, 8, 10, 12\}$. If $N$ is a new exponent, then $N\in \{5, 8, 12\}$. We will show these do not occur, and we will also rule out $N=10$.

Suppose $N \in \{5, 8, 12\}$. Then $\varphi(N)=4=2\cdot[F:\Q(j(E))]$, and by Lemma \ref{Lemma3.2} we have $T(\OO,N)=T^{\circ}(\OO,N)=2$. Thus by Lemma \ref{BabyLem} each prime dividing $N$ is ramified in $\OO$. Since $h(\OO)=p>5$, in particular the class number is odd, and so $\Delta(\OO)=-2^{\epsilon} \cdot \ell^{2a+1}$ for $\epsilon \in \{0,2\}$ and $\ell \equiv 3 \pmod 4$ prime; see, for example, Lemma 3.5 of \cite{BCS}. This shows immediately that $N \neq 5$. So suppose $N =8$. Then $\epsilon=2$, and $\OO$ is an order of conductor $2\ell^a$, where 2 is split or inert in the corresponding imaginary quadratic field $K=\Q(\sqrt{-\ell})$. Then $T(\OO,2^3) < T^{\circ}(\OO,2^3)$ by Theorem \ref{LeastDegN}, which gives a contradiction. Similarly, if $N=12$, we find  $T(\OO,2^2) < T^{\circ}(\OO,2^2)$, and so $T(\OO,12) < T^{\circ}(\OO,12)$ by Theorem \ref{LeastDegN}. 

Finally, we note $N \neq 10$, since $E$ cannot have a point of order 5 by the argument above.
\end{proof}

\begin{lem}
Let $F$ be a number field of degree $2p$ for $p>5$. Suppose $E/F$ is an $\OO$-CM elliptic curve, where $\OO$ has class number $2$. Then $E(F)_{\tors}$ has new exponent $N$ if and only if one of the following occurs:
\begin{enumerate}
\item  $N =23$, $p=11$, and $\Delta(\OO)=-115$.
\item  $N=47$, $p=23$, and $\Delta(\OO)=-235$.
\end{enumerate}
\end{lem}

\begin{proof}
Suppose $E(F)_{\tors}\cong \Z/M\Z\times\Z/N\Z$ for $M \mid N$. Note in this case $[F:\Q(j(E))] = p$, which means $\varphi(N) \mid \omega \cdot p$ by Corollary \ref{dividethmGen}. As $h(\OO)=2$, we have $\omega=2$. If $\varphi(N) \mid 2p$, then $N=2^a\cdot q^b$ where $q$ is an odd prime and $a \leq 2$, for otherwise $\ord_2(\varphi(N))>\ord_2(2p)=1$. If $b=0$, then $N\in \{1,2,4\}$, so suppose $b>0$. It follows that $a \leq 1$, for otherwise $\ord_2(\varphi(N))>1$. Thus
\[
\varphi(N)=2^{a-1}\cdot q^{b-1}(q-1) \mid 2p.
\]
 If $q=3$, then the assumption that $p>5$ implies $N=3$ or $N=6$, so suppose $q \neq 3$. Then $q-1 \mid 2p$ implies $q-1=2p$, since both $p,q$ are odd and $q \neq 3$. That is, if $E(F)$ has a point of order $N$, then $N \in \{1,2,3,4,6,2p+1, 2\cdot (2p+1)\}$ where $2p+1$ is prime. If $N$ is a new exponent, then $N \not\in \{1,2,3,4,6\}$ by definition, and so $N \in \{2p+1, 2\cdot (2p+1)\}$.

Now, suppose $E/F$ has a point $P$ of order $N=2p+1$, where $2p+1$ is prime. Then $\varphi(N)=2p$, and by Lemma \ref{Lemma3.2}, we have $p=T(\OO,N)=T^{\circ}(\OO,N)$. Thus by Lemma \ref{BabyLem}, $N$ is ramified in $\OO$. That is, $N \mid f^2 \Delta_K$. Based on the formula for $h(\OO)$ (see equation \ref{RINGCLASSDEGEQ} in $\S2$) and the classification of imaginary quadratic fields of class numbers 1 and 2 (see, for example, \cite[p.229]{cohen1}), this can happen only if $N=23$ and $\Delta(\OO) =-115$ or $N=47$ and $\Delta(\OO) = -235$. Conversely, if $\Delta(\OO) =-115$, then there exists a point of order 23 in degree $11 \cdot [\Q(j(E)):\Q]=2\cdot 11$ by Theorem \ref{LeastDegN}. Similarly, if $\Delta(\OO) =-235$, then there exists a point of order 47 in degree $23 \cdot [\Q(j(E)):\Q]=2\cdot 23$.

Finally, suppose $E/F$ has a point $P$ of order $2 \cdot (2p+1)$, where $2p+1$ is prime. Then in particular $E$ has a point of order $2p+1$, and so by the previous paragraph either $2p+1=23$ and $\Delta(\OO)=-115$ or else $2p+1=47$ and $\Delta(\OO)=-235$. In each case, 2 is inert in $\OO$, and so by Theorem \ref{bigmnthm} 
$
T(\OO,2\cdot(2p+1))=3p,
$
and we have a contradiction by Corollary \ref{dividethmGen}.
\end{proof}

\begin{lem} \label{Lem3.5}
Let $F$ be a number field of degree $2p$ for $p>5$. Suppose $E/F$ is an $\OO$-CM elliptic curve, where $\OO$ has class number $1$ and $\Delta(\OO)<-4$. Then $E(F)_{\tors}$ has new exponent $N$ if and only if we are in one of the following cases:
\begin{enumerate}
\item $N =2p+1$ where $2p+1$ is a prime split in $\OO$ and $\left( \frac{\Delta}{2} \right) = -1$.
\item $N =2 \cdot(2p+1)$ where $2p+1$ is a prime split in $\OO$ and $\left( \frac{\Delta}{2} \right) \neq -1$.
\end{enumerate}
\end{lem}

\begin{proof}
Suppose $E(F)_{\tors}\cong \Z/M\Z\times\Z/N\Z$ for $M \mid N$. Note in this case $[F:\Q(j(E))]=[F:\Q] =2p$, which means $\varphi(N) \mid \omega \cdot 2p$ by Corollary \ref{dividethmGen}. Since $\Delta<-4$, it follows that $\omega=2$.
Thus $\varphi(N) \mid 4p$. If $\varphi(N) \mid 4$, then $N \in \{1, 2, 3, 4, 5, 6, 8, 10, 12\}$. If $\varphi(N) \mid 2p$, then the proof of the previous lemma shows $N \in \{1,2,3,4,6,2p+1, 2\cdot (2p+1)\}$ where $2p+1$ is prime. Thus the only remaining case is when $\varphi(N) = 4p$. But in this case Lemma \ref{Lemma3.2} implies $T(\OO,N)=T^{\circ}(\OO,N)=2p$. By Lemma \ref{BabyLem}, if $N \geq 4$, then $N=\prod \ell_i^{a_i}$ where $\ell_i$ is a prime ramified in $\OO$ or $N=2 \cdot \prod \ell_i^{a_i}$ where 2 is split in $\OO$ and $\ell_i$ is an odd prime ramified in $\OO$. As the discriminant of $\OO$ is in \[\{-7, -8, -11, -12, -16, -19,  -27, -28,   -43, -67, -163\},\] there are no possibilities such that $\varphi(N)=4p$ for $p>5$. We note that if $N$ is new, then $N\not \in \{1,2, 3,4,6, 10\}$ by definition. Furthermore, $N \not\in \{5,8,12\}$, for otherwise $4 \mid [F:\Q]$; see, for example, the table in the appendix of \cite{BCS}. 

Now, suppose $N=2p+1$, where $2p+1$ is prime. Since $p>5$, we see immediately from the list of imaginary quadratic discriminants of class number 1 that $N$ is not ramified in $\OO$, and $N$ is not inert, for otherwise $T(\OO,N)=2p(p+1) \nmid [F:\Q]$ by Theorem \ref{bigmnthm}. Now, suppose $N$ is split in $\OO$. Then $T^{\circ}(\OO,N)=2p$ by Theorem \ref{LeastDegN}. By Theorem \ref{bigmnthm}, $N=2 \cdot (2p+1)$ is possible only if 2 is split or ramified in $\OO$. Conversely, suppose $2$ is split or ramified in $\OO$. Then $\Delta \in \{-7, -8, -12, -16, -28\}$. In each case, such an $\OO$-CM elliptic curve $E/F$ will always have an $F$-rational point of order 2; this can be seen, for example, by the fact that any model of such an elliptic curve over $\Q$ will have a rational point of order 2, and points of order 2 are invariant under quadratic twist.
\end{proof}

\begin{lem} \label{Lem3.6}
Let $F$ be a number field of degree $2p$ for $p>5$. Suppose $E/F$ is an $\OO$-CM elliptic curve, where $\Delta(\OO)=-4$. Then $E(F)_{\tors}$ has new exponent $N$ if and only if $N =2\cdot(4p+1)$ where $4p+1$ is prime.
\end{lem}

\begin{proof}
Suppose $E(F)_{\tors}\cong \Z/M\Z\times\Z/N\Z$ for $M \mid N$. If $\Delta=-4$, then $\omega=4$ and $\varphi(N) \mid 8p$ by Corollary \ref{dividethmGen}. If $\varphi(N) \mid 8$, then \[N \in \{1, 2, 3, 4, 5, 6, 8, 10, 12, 15, 16, 20, 24, 30\}.\] If $\varphi(N) \mid 4p$, then $\ord_2(\varphi(N)) \leq \ord_2(4p)=2$ implies $N=2^a\cdot q_1^b \cdot q_2^c$ where $q_1,q_2$ are odd primes and $a \leq 3$. If $a=3$, then $N=8$, so suppose $a=2$. Then $N=2^2 \cdot q_1^{b}$. If $b>0$, then the assumption that $\varphi(N)=2\cdot q_1^{b-1}(q_1-1) \mid 4p$ implies $q_1=3$ or $2p+1$ as above, and $b=1$. If $a \leq1$, then $N=2^a\cdot q_1^b \cdot q_2^c$, and we have
\[
\varphi(N)=q_1^{b-1}(q_1-1) \cdot q_2^{c-1}(q_2-1) \mid 4p.
\]
In particular, $q_i-1 \mid 4p$ implies $q_i \in \{3,5,2p+1, 4p+1\}$, since it is an odd prime. Thus if $\varphi(N) \mid 4p$, the only possibilities are 
\[N \in \{1, 2, 3, 4, 5, 6, 8, 10, 12,2p+1, 2\cdot(2p+1),3\cdot(2p+1), 4 \cdot(2p+1),6 \cdot(2p+1), 4p+1, 2\cdot(4p+1)\},\] where $2p+1$ and $4p+1$ can arise only if they are prime. Finally, suppose $\varphi(N)=8p$. But Lemma \ref{Lemma3.2} implies $2p=T(\OO,N)=T^{\circ}(\OO,N)$. By Lemma \ref{BabyLem}, $N \in\{1,3,2^a\}$ since $\Delta=-4$, but none of these satisfy $\varphi(N)=8p$.

We note that if $N$ is a new exponent, then $N\not\in \{1,2, 3,4,6, 10\}$ by definition, so we may remove these values from consideration. By Theorem \ref{bigmnthm}, $T(\OO,N) \nmid 2p$ if $N \in \{8,12, 15, 20\}$, which implies $N$ cannot be any of these values, along with 16, 24, or 30. Though there can exist a point of order 5 on an $\OO$-CM elliptic curve defined over a number field $F$ of degree $2p$, such an elliptic curve corresponds to an equation of the form $y^2=x^3+Ax$ and so has an $F$-rational point of order 2. Thus an exponent of 5 is not possible. Now, consider a prime $N=2p+1$, which cannot be ramified since $\Delta=-4$. If $N$ is inert, then $T(\OO,N)=p(p+1) \nmid 2p$. In addition, $N$ cannot be split, since then $T(\OO,N)=p/2$ would not be an integer. Thus if $E(F)_{\tors}$ has new exponent $N$, then $N \in \{4p+1, 2\cdot(4p+1)\}$ where $4p+1$ is prime. Since $\left( \frac{-4}{4p+1} \right)=1$, there is a point of order $4p+1$ in degree $2p$ by Theorem \ref{LeastDegN} and Theorem \ref{bigmnthm}. As $E$ has the form $y^2=x^3+Ax$, there is a point of order 2 as well, so $N=2 \cdot (4p+1)$.
\end{proof}

\begin{lem}
Let $F$ be a number field of degree $2p$ for $p>5$. Suppose $E/F$ is an $\OO$-CM elliptic curve, where $\Delta(\OO)=-3$. Then $E(F)_{\tors}$ has new exponent $N$ if and only if we are in one of the following cases:
\begin{enumerate}
\item $N=49$ and $p=7$.
\item $N =6p+1$ where $6p+1$ is prime.
\end{enumerate}
\end{lem}
\begin{proof}
Suppose $E(F)_{\tors}\cong \Z/M\Z\times\Z/N\Z$ for $M \mid N$. If $\Delta=-3$, then $\omega=6$ and $\varphi(N) \mid 12p$ by Corollary \ref{dividethmGen}. If $\varphi(N) \mid 12$, then 
\[
N \in \{1, 2, 3, 4, 5, 6, 7, 8, 9, 10, 12, 13, 14, 18, 21, 26, 28, 36, 42 \}.
\]
If $\varphi(N) \mid 6p$, then $\ord_2(\varphi(N)) \leq \ord_2(6p)=1$ implies $N=2^a\cdot q^b$ for an odd prime $q$ and $a \leq 2$. Suppose $b>0$. Then $a \leq 1$ and $q-1 \mid 6p$ implies $q \in \{3,7, 2p+1, 6p+1\}$ since $q$ is an odd prime. If $\varphi(N) \mid 4p$, then as shown in the proof of Lemma \ref{Lem3.6},
\[N \in \{1, 2, 3, 4, 5, 6, 8, 10, 12,2p+1, 2\cdot(2p+1),3\cdot(2p+1), 4 \cdot(2p+1),6 \cdot(2p+1), 4p+1, 2\cdot(4p+1)\},\] where $2p+1$ and $4p+1$ can arise only if they are prime. 
Finally, suppose $\varphi(N)=12p$. But then Lemma \ref{Lemma3.2} implies $T(\OO,N)=T^{\circ}(\OO,N)=2p$. Since $\Delta=-3$, Lemma \ref{BabyLem} implies $N \in \{1,2,3^a\}$, but none of these satisfy $\varphi(N)=12p$.

If $N$ is a new exponent, then $N\not\in \{1,2, 3,4,6, 7, 10\}$ by definition. By Theorem \ref{bigmnthm}, $N\not\in\{5, 8, 9, 12, 14, 18, 26, 28, 36, 42, 98\}$ since $T(\OO,N) \nmid 2p$. Next, we will show $2p+1 \nmid N$ when $2p+1$ is prime. Since $\Delta =-3$, $2p+1$ is not ramified, and it cannot be split because then $T(\OO,2p+1)\notin \Z$. Thus $2p+1$ is inert in $\OO$, and $T(\OO,N)>2p$; contradiction. Similarly, we cannot have $4p+1 \mid N$ when $4p+1$ is prime.

The remaining options are $N \in \{13, 21,  49, 6p+1, 2 \cdot (6p+1)\}$ where $6p+1$ is prime. To see $N \neq 13$, note that by Lemma 7.6 and Theorem 7.8 in \cite{BC1}, if $P \in E$ has order 13 and $K=\Q(\sqrt{-3})$, then $[K(\mathfrak{h}(P)):K]=2$ or 24, where $\mathfrak{h}$ denotes a Weber function on $E$. Since $P$ is defined over a number field of degree $2p$, it must be that $[K(\mathfrak{h}(P)):K]=2$. Then $[\Q(\mathfrak{h}(P)):\Q]=2$, since its degree must also divide $2p$. However, then there is a twist of $E$ defined over $\Q(\mathfrak{h}(P))$ such that $P$ becomes rational, and $T(\OO,13)=T^{\circ}(\OO,13)=2$. This contradicts Theorem \ref{LeastDegN}. Similarly, $N \neq 21$: by Lemma 7.6, Proposition 7.7, and Theorem 7.8 in \cite{BC1}, $[K(\mathfrak{h}(P)):K]=[\Q(\mathfrak{h}(P)):\Q]=2$ since this quantity must divide $2p$. But then $T(\OO,21)=T^{\circ}(\OO,21)=2$, which contradicts Theorem \ref{LeastDegN}. 

We note $N=49$ does occur as a new exponent in degree $2\cdot 7$ by Theorem \ref{LeastDegN}, and this is the only possible degree since $\varphi(49) \mid 12p$ only if $p=7$. If $N=6p+1$ is prime, then $\left( \frac{-3}{6p+1} \right)=1$, and there exists a point of order $N$ in degree $2p$ by Theorem \ref{LeastDegN}. However, 
$T(\OO,2 \cdot (6p+1))=3p,
$ 
and so we cannot have an $\OO$-CM elliptic curve with a point of order $2 \cdot (6p+1)$ in degree $2p$.
\end{proof}

\section{Determining New Torsion Subgroups}

Suppose $F$ is a number field of degree $2p$, where $p>5$ is prime, and $E/F$ is an $\OO$-CM elliptic curve. If $E(F)_{\tors} \cong \Z/M\Z \times \Z/N\Z$ for $M \mid N$ is new, then either $N$ occurs already as an exponent of a CM torsion subgroup in degree 1 or 2 and
\[
N \in \{1,2,3, 4,6,7, 10\}
\]
by \cite[$\S4.1, 4.2$]{tor2}, or else by the previous section we are in one of the following cases:
\begin{enumerate}
\item $\Delta(\OO)=-115$, $p=11$, and $N=23$,
\item $\Delta(\OO)=-235$, $p=23$, and $N=47$,
\item $\Delta(\OO) \in\{-11,-19,-27,-43,-67,-163\}$ and $N=2p+1$ is prime with $\left( \frac{\Delta}{2p+1} \right) = 1$,
\item $\Delta(\OO) \in \{-7,-8, -12, -16, -28\}$ and $N=2 \cdot(2p+1)$ where $2p+1$ is prime with $\left( \frac{\Delta}{2p+1} \right) = 1$,
\item $\Delta(\OO)=-4$ and $N=2 \cdot (4p+1)$ where $4p+1$ is prime,
\item $\Delta(\OO)=-3$, $p=7$, and $N=49$,
\item $\Delta(\OO)=-3$ and $N=6p+1$ where $6p+1$ is prime.
\end{enumerate}

\begin{lem}
Suppose $F$ is a number field of degree $2p$, where $p>5$ is prime, and $E/F$ is an $\OO$-CM elliptic curve. If $E(F)_{\tors} \cong \Z/M\Z \times \Z/N\Z$ for $M \mid N$ is new, then $M=1$ or 2.
\end{lem}

\begin{proof} Suppose $\ell \mid M$ is prime. If $\ell=\omega \cdot p+1 >4$, then by Theorem \ref{bigmnthm} we have
$
T(\OO,\ell,\ell)>2p.
$
This is a contradiction. By $\S3$ as summarized above, it remains to consider 
\[
M \in \{3,4,5,6, 7, 10, 49\}.
\]

Note that for any $M \geq 3$, the CM field $K$ is contained in $F(E[M])$ by Lemma 3.15 of \cite{BCS}, and so $2 \cdot T(\OO,M,N) \mid [F:\Q]$ by Theorem \ref{bigmnthm}. We reach a contradiction for 
\[
(M,N) \in \{ (3,6), (4,4),(5,10), (6,6),  (7,7), (7,49), (10,10), (49,49)\}.
\]
This leaves only $(M,N)=(3,3)$, but $\Z/3\Z \times \Z/3\Z$ occurs already in degree 2 by \cite[$\S4.2$]{tor2}.
\end{proof}

By the classification of CM torsion subgroups in degree 2 \cite[$\S4.2$]{tor2} and the previous lemma, the only possible new subgroup with an old exponent is $\Z/2 \Z \times \Z/10\Z$. Any other new torsion subgroup will be of the form $\Z/N\Z$ or $\Z/2\Z \times \Z/N\Z$ for a new exponent $N$. In particular, if $N$ is odd, then the new torsion subgroup is precisely $\Z/N\Z$. It remains to check whether one can have full 2-torsion in each of the following cases:
\begin{enumerate}
\item $N=10$
\item $N=2 \cdot(2p+1)$ where $2p+1$ is prime, $\Delta(\OO) \in \{-7,-8, -12, -16, -28\}$ and $\left( \frac{\Delta}{2p+1} \right) = 1$
\item $N=2 \cdot (4p+1)$ where $4p+1$ is prime and $\Delta(\OO)=-4$
\end{enumerate}

\begin{lem}
$\Z/2 \Z \times \Z/10 \Z$ does not occur as a new torsion subgroup of a CM elliptic curve defined over a number field of degree $2p$ for $p>5$.
\end{lem}

\begin{proof}
Suppose $E/F$ is an $\OO$-CM elliptic curve with $E(F)_{\tors} \cong \Z/2 \Z \times \Z/10 \Z$, and let $P \in E(F)$ be a point of order 10. By Corollary \ref{dividethm}, we have $2 \mid T(\OO,10)$ unless $\Delta =-4$. Thus $2 \nmid h(\OO)$, for otherwise $4 \mid [F:\Q]$ by Corollary \ref{dividethmGen}.  In addition, $h(\OO) \neq p$ by Lemma \ref{Lem3.3}. Since $h(\OO) \mid 2p$, it follows that $h(\OO)=1$. Moreover, $\Delta(\OO)=-4$ by the table in the appendix of \cite{BCS} since otherwise $[\Q(\mathfrak{h}(P)):\Q] \nmid 2p$. Also, this table shows that $\Q(\mathfrak{h}(P))$ has degree 2, as neither 4 nor 8 divide $2p$. Since $T(\OO,10)=1$ by Theorem \ref{bigmnthm}, it follows that $K=\Q(\sqrt{-1})=\Q(\mathfrak{h}(P))$. In particular, $K \subseteq F$. Moreover,
$
T(\OO,2,10)=2.
$
Since $T(\OO,2,10) \mid [F:K]$, it follows that $4 \mid [F:\Q]$; contradiction.
\end{proof}

\begin{lem}
$\Z/2 \Z \times \Z/2(2p+1) \Z$ where $2p+1$ is prime and $p>5$ is prime occurs as a new torsion subgroup of an $\OO$-CM elliptic curve in degree $2p$ if and only if $\Delta(\OO)=-7$ and $\left( \frac{\Delta}{2p+1} \right) = 1$.

\end{lem}

\begin{proof}
Suppose $E/F$ is an $\OO$-CM elliptic curve with $E(F)_{\tors} \cong \Z/2 \Z \times \Z/2(2p+1) \Z$. Then $\Delta(\OO) \in \{-7,-8, -12, -16, -28\}$ and $\left( \frac{\Delta}{2p+1} \right) = 1$ by Lemma \ref{Lem3.5}.
We consider two cases. First, suppose $2$ is ramified in $\OO$. Then $T(\OO,2,2(2p+1))=2p$ by Theorem \ref{bigmnthm}, yet by Theorem \ref{LeastDegN}, $T^{\circ}(\OO,2,2(2p+1))=2 \cdot 2p$ since $2p+1$ is split. So we must have 2 split in $\OO$, which occurs if and only if $\Delta=-7$. Then $T(\OO,2,2(2p+1))=p$, and $T^{\circ}(\OO,2,2(2p+1))=2p$, as desired.

Finally, we must show that if $\Delta(\OO)=-7$ and $\Z/2(2p+1) \Z \hookrightarrow E(F)_{\tors}$, then in fact $E$ has full 2-torsion over $F$. By Lemma 7.6 and Theorem 7.8 in \cite{BC1}, if $P \in E(F)$ has order $2p+1$ and $K=\Q(\sqrt{-7})$, then $[K(\mathfrak{h}(P)):K]=p$ or $2p^2$. Since $P$ is defined over a number field of degree $2p$, it must be that $[K(\mathfrak{h}(P)):K]=p$. By Theorem  \ref{LeastDegN}, $[\Q(\mathfrak{h}(P):\Q]=2p$, and so $K \subseteq \Q(\mathfrak{h}(P)) \subseteq F$. Thus $E$ has full 2-torsion over $F$ by Theorem 4.2 of \cite{BCS}, as we recall that 2-torsion is model-independent.
\end{proof}

\begin{lem}
$\Z/2 \Z \times \Z/2(4p+1) \Z$ where $4p+1$ is prime and $p>5$ is prime does not occur as a new torsion subgroup of a CM elliptic curve in degree $2p$.
\end{lem}

\begin{proof}
Suppose $F$ is a number field of degree $2p$ for $p>5$ prime, and suppose $E/F$ is an $\OO$-CM elliptic curve with a point of order $4p+1$, where $4p+1$ is prime. Then by the lemmas of $\S3.2$, $\Delta(\OO)=-4$. Since $4p+1$ is split in $\OO$, Theorem \ref{bigmnthm} implies $T(\OO,2,2(2p+1))=2p$. However, by Theorem \ref{LeastDegN}, $T^{\circ}(\OO,2,2(4p+1))=2 \cdot 2p$, and we have a contradiction.
\end{proof}

\bibliographystyle{amsplain}
\bibliography{bibliography1}
\end{document}